\documentstyle[twoside, amsmath, amssymb, amsfonts, 12pt]{article}

\newtheorem{theorem}{Theorem}[section]
\newtheorem{lemma}[theorem]{Lemma}

\newtheorem{proposition}[theorem]{Proposition}
\newtheorem{corollary}[theorem]{Corollary}

\newenvironment{proof}{\begin{trivlist} \item[]{\em Proof.}}{\end{trivlist}}
\newcount\refno
\newcommand\be{\begin{equation}}
\newcommand\ee{\end{equation}}
\newcommand\bn{\begin{eqnarray}}
\newcommand\en{\end{eqnarray}}
\newcommand\bns{\begin{eqnarray*}}
\newcommand\ens{\end{eqnarray*}}

\newcommand\bR{{\mathbb R}}
\newcommand\bN{{\mathbb N}}

\newcommand\bZ{{\mathbb Z}}

\newcommand{\seqnum}[1]{{\underline{#1}}}

\def\eop{\hfill\rule{2.0mm}{2.0mm}}

\def\a{\alpha}

\title{ Use Impulse Response Sequences in the Construction of Number Sequence Identities}

\author{Tian-Xiao He \\
{\small Department of Mathematics }\\
 {\small Illinois Wesleyan University}\\
 {\small Bloomington, IL 61702-2900, USA}\\
}

\date{}

\begin{document}

\maketitle
\setcounter{page}{1}
\pagestyle{myheadings}
\markboth{T. X. He }
{Sequence Identities}

\begin{abstract}
\noindent We define impulse response sequence in the set of all linear recurring sequences satisfying a linear recurrence relation of order $r$. The generating function and expression of the impulse response sequence are presented. Some identities of impulse response sequences including a type of nonlinear expressions are established. The interrelationship between the impulse response sequence and other linear recurring sequences in the same set is given, which is used to transfer the  identities of impulse response sequences to those of the linear recurring sequences in the same set. Some applications of impulse response sequences to the structure of Stirling numbers of the second order, the Wythoff array, and the Boustrophedon transform are studied.

\vskip .2in
\noindent
AMS Subject Classification: 05A15, 05A19, 11B39, 11B73, 11B75, 11Y55. 

\vskip .2in
\noindent
{\bf Key Words and Phrases:} generating function, number sequence, linear recurrence relation, Stirling numbers of the second kind, Wythoff array, Boustrophedon transform, Fibonacci numbers, Pell numbers, Jacobsthal numbers, Lucas numbers. 
\end{abstract}

\section{Introduction}

Many number and polynomial sequences can be defined, characterized, evaluated, and classified by linear recurrence relations with certain orders. A number sequence $\{ a_n\}$ is called sequence of order $r$ if it satisfies the linear recurrence relation of order $r$
\begin{equation}\label{eq:1}
a_n= \sum^r_{j=1}p_j a_{n-j}, \quad n\geq r, 
\end{equation}
for some constants $p_j$ ($j=1,2,\ldots, r$), $p_r\not= 0$, and initial conditions $a_j$ ($j=0,1,\ldots, r-1$). Linear recurrence relations with constant coefficients are important in subjects including combinatorics, pseudo-random number generation, circuit design, and cryptography, and they have been studied extensively. To construct an explicit formula of the general term of a number sequence of order $r$,  one may use generating function, characteristic equation, or a matrix method (See Comtet \cite{Com},  Hsu \cite{HSU83}, Niven, Zuckerman, and Montgomery \cite{NZM}, Strang \cite{Str},  Wilf \cite{Wilf}, etc.) Let $A_{r}$ be the set of all linear recurring sequences defined by the homogeneous linear recurrence relation (\ref{eq:1}) with coefficient set $E_{r}=\{ p_{1}, p_{2}, \ldots, p_r\}$. To study the structure of $A_r$ with respect to $E_r$, we define the impulse response sequence (or, IRS for abbreviation) in $A_r$, which is a particular sequence in $A_r$ with initials $a_0=a_1=\cdots =a_{r-2}=0$ and $a_{r-1}=1$. 

In next section, we will give the generating function and the expression of the IRS and find out the interrelationships between the IRS and the sequences in the same set $A_r$. A numerous examples of IRS in sets $A_2$ and $A_3$ will be giving in Section $3$. In \cite{HS}, Shiue and the authors presented a method of the construction of general term expressions and identities for the linear recurring sequences of order $r=2$ by using the reduction order method. However, the reduction method is too complicated for the linear recurring sequences of order $r>2$. In Sections $2$ and $3$, we will see that our method by using IRS can be applied to linear recurring sequences with arbitrary order readily in the construction of their general term expressions and identities. In Section $4$, with using the symbol method shown in \cite{Hsu}, we derive a type of identities of IRS in $A_2$ including a type of nonlinear expressions. The relationship between the IRS and other linear recurring sequences in the same set is used to transfer the identities of IRS to those of the linear recurring sequences in the same set. Finally, we present some applications of IRS in the discussion of structure of Stirling numbers of the second kind, the Wythoff array, and the Boustrophedon transform.

\section{Impulse response sequences}
Among all the homogeneous linear recurring sequences satisfying $r$th order homogeneous linear recurrence relation (\ref{eq:1}) 
with a nonzero $p_{r}$ and arbitrary initials $\{ a_{j}\}^{r}_{j=0}$, we define the impulse response sequence (IRS) with respect to $E_{r}=\{ p_{j}\}^{r}_{j=1}$ as the sequence with initials $a_{0}=a_{r-2}=0$ and $a_{r-1}=1$. In particular, when  $r=2$, the homogeneous linear recurring sequences with respect to $E_{2}=\{ p_{1}, p_{2}\}$ satisfy 

\be\label{eq:5.1}
a_n= p_{1}a_{n-1}+p_{2} a_{n-2}, \quad n\geq 2
\ee
with arbitrary initials $a_{0}$ and $a_{1}$, or, initial vector $(a_0, a_1)$. If the initial vector $(a_{0}, a_{1})=(0,1)$, the corresponding sequence generated by using (\ref{eq:5.1}) is the impulse response sequence with respect to $E_{2}$. For instance, Fibonacci sequence $\{ F_{n}\}$ is the IRS with respect to $\{ 1,1\}$, Pell number sequence $\{ P_{n}\}$ is the IRS with respect to $\{ 2,1\}$, and Jacobathal number sequence $\{ J_n\}$ is the IRS with respect to $\{1,2\}$. 

In the following, we will present the structure of the linear recurring sequences defined by (\ref{eq:1}) using their characteristic polynomial. Then, we may find the interrelationship of those sequences with the impulse response sequence. 

Let $A_{r}$ be the set of all linear recurring sequences defined by the homogeneous linear recurrence relation (\ref{eq:1}) with coefficient set $E_{r}=\{ p_{1}, p_{2}, \ldots, p_r\}$, and let $\tilde F^{( r)}_{n}$ be the IRS of $A_{r}$, namely,  $\tilde F^{(r)}_n$ satisfies (\ref{eq:1}) with initials $\tilde F^{(r)}_0=\cdots =\tilde F^{(r)}_{r-2}=0$ and $\tilde F^{(r)}_{r-1}=1$. 

Suppose $\{ a_n\}\in A_r$. Let $B=|p_1|+|p_2|+\cdots +|p_r|$, and for $n\in {\bN}_0$ let $M_n=\max\{ |a_0|, |a_1|, \ldots, |a_n|\}$. Thus, for $n\geq r$, $M_n\leq BM_{n-1}$. By induction, there holds $|a_n|\leq C B^n$ for some constant $A$ and $n=0,1,2,\ldots.$ A sequence satisfying a bound of this kind is called a sequence having at most exponential growth (see\cite{NZM}). Therefore, the generating function $\sum_{j\geq 0} a_j t^j$ of such sequence has a positive radius of convergence, $0< c< 1/B$. More precisely, if $|t|\leq c$, then, by comparison with the convergent geometric series $\sum_{j\geq 0} CB^j c^j$, the series $\sum_{j\geq 0} a_j t^j$ converges absolutely. In the following, we always assume that $\{ a_n\}$ have at most exponential growth.

\begin{proposition}\label{pro:3.1}
Let $\{ a_{n}\}\in A_{r}$, i.e., let $\{ a_{n}\}$ be the linear recurring sequence defined by (\ref{eq:1}). Then its generating function $P(t)$ can be written as

\be\label{eq:3.1}
P_r(t)=\{ a_{0}+\sum^{r-1}_{n=1}\left( a_{n}-\sum^{n}_{j=1}p_{j}a_{n-j}\right)t^{n}\}/\{1-\sum^{r}_{j=1}p_{j}t^{j}\}.
\ee
Hence, the generating function for the IRS with respect to $\{ p_j\}$ is 

\be\label{eq:3.2}
\tilde P_r(t)=\frac{ t^{r-1}}{1-\sum^{r}_{j=1}p_{j}t^{j}}.
\ee
\end{proposition}

\begin{proof}
Multiplying $t^{n}$ on the both sides of (\ref{eq:1}) and summing from $n=r$ to infinite yields 

\bns
&&\sum_{n\geq r}a_{n}t^{n}=\sum_{n\geq r}\sum^{r}_{j=1}p_{j}a_{n-j}t^{n}=\sum^{r}_{j=1}p_{j}\sum_{n\geq r}a_{n-j}t^{n}\\
&=&\sum^{r-1}_{j=1}p_{j}\left[ \sum_{n\geq j}a_{n-j}t^{n}-\sum^{r-1}_{n=j}a_{n-j}t^{n}\right]+p_{r}\sum_{n\geq r}a_{n-r}t^{n}\\
&=&\sum^{r-1}_{j=1}p_{j}\left[ t^{j}\sum_{n\geq 0}a_{n}t^{n}-\sum^{r-1}_{n=j}a_{n-j}t^{n}\right] +p_{r}t^{r}\sum_{n\geq 0}a_{n}t^{n}\\
\ens
Denote $P(t)=\sum_{n\geq 0}a_{n}t^{n}$. Then the equation of the leftmost and the rightmost sides can be written as 

\[
P_r(t)-\sum^{r-1}_{n=0}a_{n}t^{n}=P(t)\sum^{r-1}_{j=1}p_{j}t^{j} -\sum^{r-1}_{j=1}\sum^{r-1}_{n=j}p_{j}a_{n-j}t^{n}+p_{r}t^{r}P(t),
\]
which implies 

\[
P_r(t)\left( 1-\sum^{r}_{j=1}p_{j}t^{j}\right)=\sum^{r-1}_{n=0}a_{n}t^{n}-\sum^{r-1}_{n=1}\sum^{n}_{j=1}p_{j}a_{n-j}t^{n}.
\]
Therefore, (\ref{eq:3.2}) follows immediately. By substituting $a_{0}=\cdots =a_{r-2}=0$ and $a_{r-1}=1$ into (\ref{eq:3.1}), we obtain (\ref{eq:3.2}).
\end{proof}
\eop

We now give the explicit expression of $\tilde F^{(r)}_n$ in terms of the roots of the characteristic polynomial of recurrence relation shown in (\ref{eq:1}).

\begin{theorem}\label{thm:1.3}
Sequence $\{\tilde F^{(r)}_n\}_n$ is the IRS with respect to $\{p_j\}$ defined by (\ref{eq:1}) with $\tilde F^{( r)}_{0}=\tilde F^{( r)}_{1}=\cdots =\tilde F^{( r)}_{r-2}=0$ ane $\tilde F^{(r)}_{r-1}=1$ if and only if there holds 

\be\label{eq:5.7}
\tilde F^{(r)}_n =\sum^\ell_{j=1}\frac{ {n\choose m_j-1}\alpha^{n-m_j+1}_j}{\Pi^\ell_{k=1, k\not= j}(\alpha_j-\alpha_k)}
\ee
for $n\geq r$,  where $\alpha_1,\alpha_2,\ldots, \alpha_\ell$ are the distinct real or complex roots of characteristic polynomial,  $p(t)=t^r-p_1t^{r-1}-\cdots -p_r$, of the recurrence relation (\ref{eq:1}) with the multiplicities $m_1,m_2,\ldots, m_\ell$, respectively, and $m_1+m_2+\cdots +m_\ell=r$. In particular, if all $\alpha_1,\alpha_2, \ldots$, $\alpha_r$ are of multiplicity one,  then (\ref{eq:5.7}) becomes to 

\be\label{eq:5.7.2}
\tilde F^{(r)}_n=\sum^r_{j=1}\frac{\alpha^n_j}{\Pi^r_{k=1,k\not= j} (\alpha_j-\alpha_k)}, \quad n\geq r,
\ee
\end{theorem}

\begin{proof} First, we prove (\ref{eq:5.7.2}). Then we use it to prove the necessity 
of (\ref{eq:5.7}). 
Let $\alpha_{j}$,  $j=1,2,\ldots,\ell$, be $\ell$ roots of the characteristic polynomial $t^{n}-\sum^{r}_{i=1}p_{i}t^{n-i}$. Then from (\ref{eq:3.2}) 

\bn\label{eq:5.7.3}
&&\tilde P_r(t)=\frac{t^{r-1}}{1-\sum^r_{j=1}p_j t^j}=\frac{1/t}{(1/t)^r-\sum^r_{j=1}p_j(1/t)^{n-j}}\nonumber \\
&=&\frac{1/t}{\Pi^r_{j=1}(1/t-\alpha_j)}=\frac{t^{r-1}}{\Pi^r_{j=1} (1-\alpha_j t)}.
\en
We now prove (\ref{eq:5.7.2}) from (\ref{eq:5.7.3}) by using mathematical induction for $r$ under the assumption of that all solutions $\alpha_j$ are distinct. 
Because $[t^n]\tilde P_r(t)=\tilde F^{(r)}_n$, we need to show 

\be\label{eq:5.7.4}
[t^n]\tilde P_r(t)=[t^n] \frac{t^{r-1}}{\Pi^r_{j=1} (1-\alpha_j t)}=\sum^r_{j=1}\frac{\alpha^n_j}{\Pi^r_{k=1,k\not= j} (\alpha_j-\alpha_k)}, \quad n\geq r.
\ee
(\ref{eq:5.7.4}) is obviously true for $r=1$. Assume (\ref{eq:5.7.4}) holds for $r=m$. We find 

\[
[t^n]\tilde P_{m+1}(t)=[t^n] \frac{t^{m}}{\Pi^{m+1}_{j=1} (1-\alpha_j t)},
\]
which implies 

\[
[t^n] (1-\alpha_{m+1}t)\tilde P_{m+1}(t)=[t^{n-1}] \frac{t^{m-1}}{\Pi^{m}_{j=1} (1-\alpha_j t)}
=\sum^m_{j=1}\frac{\alpha^{n-1}_j}{\Pi^m_{k=1,k\not= j} (\alpha_j-\alpha_k)}
\]
because of the induction assumption. Noting $[t^n] \tilde P_{m+1}(t)=\tilde F^{(m+1)}_n$, from the last equations one may write 

\be\label{eq:5.7.5}
\tilde F^{(m+1)}_n -\alpha_{m+1}\tilde F^{(m+1)}_{n-1}=\sum^m_{j=1}\frac{\alpha^{n-1}_j}{\Pi^m_{k=1,k\not= j} (\alpha_j-\alpha_k)},
\ee
which has the solution

\be\label{eq:5.7.6}
\tilde F^{(m+1)}_n= \sum^{m+1}_{j=1}\frac{\alpha^{n}_j}{\Pi^{m+1}_{k=1,k\not= j} (\alpha_j-\alpha_k)}.
\ee
Indeed, we have 

\bns
&&\alpha_{m+1}\tilde F^{(m+1)}_{n-1}+\sum^m_{j=1}\frac{\alpha^{n-1}_j}{\Pi^m_{k=1,k\not= j} (\alpha_j-\alpha_k)}\\
&=&\alpha_{m+1}\sum^{m+1}_{j=1}\frac{\alpha^{n-1}_j}{\Pi^{m+1}_{k=1,k\not= j} (\alpha_j-\alpha_k)}+\sum^m_{j=1}\frac{\alpha^{n-1}_j}{\Pi^m_{k=1,k\not= j} (\alpha_j-\alpha_k)}\\
&=&\sum^m_{j=1}\frac{\alpha^{n-1}_j}{\Pi^{m+1}_{k=1,k\not= j} (\alpha_j-\alpha_k)}\left( \alpha_{m+1}+\alpha_j-\alpha_{m+1}\right) +\frac{\alpha^{n}_{m+1}}{\Pi^{m}_{k=1} (\alpha_{m+1}-\alpha_k)}\\
&=&\sum^{m+1}_{j=1}\frac{\alpha^{n}_j}{\Pi^{m+1}_{k=1,k\not= j} (\alpha_j-\alpha_k)}=\tilde F^{(m+1)}_n,
\ens
and (\ref{eq:5.7.6}) is proved. 

To prove (\ref{eq:5.7}), it is sufficient to show it holds for the case of one multiple root, say $\alpha_i$ with the multiplicity $m_i$. The formula of $\tilde F^{(r)}_n$ in this case can be derived from (\ref{eq:5.7.2}) by using the limit process. More precisely, let $m_i$ roots of the characteristic polynomial, denoted by $\alpha_{i_j}$ ($1\leq j\leq m_i$), approach to $\alpha_i$, then (\ref{eq:5.7.2}) will be reduced to (\ref{eq:5.7}) with respect to $\alpha_i$. Indeed, taking the limit as $\alpha_{i_j}$ ($1\leq j\leq m_i$) approaching to $\alpha_i$ on the both sides of (\ref{eq:5.7.2}) yields 

\bn
&&\lim_{(\alpha_{i_1}, \ldots, \alpha_{i_{m_i}})\to (\alpha_i, \ldots, \alpha_i)}\tilde F^{(r)}_n\nonumber\\
&=&\frac{1}{\Pi^{r-m_i+1}_{k=1,k\not= i}(\alpha_i-\alpha_k)}\lim_{(\alpha_{i_1}, \ldots, \alpha_{i_{m_i}})\to (\alpha_i, \ldots, \alpha_i)}\left[ \frac{\alpha^n_{i_1}}{\Pi^{m_i}_{k=2}(\alpha_{i_1}-\alpha_{i_k})}+\frac{\alpha^n_{i_2}}{\Pi^{m_i}_{k=1,k\not=2}(\alpha_{i_2}-\alpha_{i_k})}\right.\nonumber\\
&&\qquad \quad \left. +\cdots+\frac{\alpha^n_{i_{m_i}}}{\Pi^{m_i}_{k=1,k\not= m_i}(\alpha_{i_{m_i}}-\alpha_{i_k})}\right] +\sum^r_{j=1, j\not= i}\frac{\alpha^n_j}{\Pi^r_{k=1,k\not= j} (\alpha_j-\alpha_k)}.\label{eq:5.7.7}
\en
It is obvious that the summation in the bracket of equation (\ref{eq:5.7.7}) is the expanded form of the divided difference of the function $f(t)=t^n$, denoted by $f[\alpha_{i_1},\alpha_{i_2},\ldots, \alpha_{i_{m_i}}]$, at nodes $\alpha_{i_1}, \alpha_{i_2},\ldots, \alpha_{i_{m_i}}$. Using the mean value theorem for divided difference, one may obtain 

\[
\lim_{(\alpha_{i_1}, \ldots, \alpha_{i_{m_i}})\to (\alpha_i, \ldots, \alpha_i)}f[\alpha_{i_1},\alpha_{i_2},\ldots, \alpha_{i_{m_i}}]=\frac{f^{(m_i-1)}(\alpha_i)}{(m_i-1)!}={n\choose {m_i-1}} \alpha^{n-m_i+1}_i.
\]
Therefore, by taking limit $\alpha_{i_j}\to \alpha_i$ for $1\leq j\leq m_i$ on the both sides of (\ref{eq:5.7.7}), we obtain 

\[
\lim_{(\alpha_{i_1}, \ldots, \alpha_{i_{m_i}})\to (\alpha_i, \ldots, \alpha_i)}\tilde F^{(r)}_n= \frac{{n\choose {m_i-1}}\alpha_i^{n-m_1+1}}{\Pi^{r-m_i+1}_{k=1,k\not= i}(\alpha_i-\alpha_k)}+ \sum^r_{j=1, j\not= i}\frac{\alpha^n_j}{\Pi^r_{k=1,k\not= j} (\alpha_j-\alpha_k)},
\]
which implies the correction of (\ref{eq:5.7}) for the multiple root $\alpha_i$ with multiplicity $m_i$. Therefore, (\ref{eq:5.7}) follows by taking the same process for each 
multiple root of the characteristic polynomial $P_r(t)$ shown in Proposition \ref{pro:3.1}.

Finally, we prove the sufficiency. For $n\geq r$, 

\bns
&&\sum^{r}_{i=1}p_{i}\tilde F^{( r)}_{n-i}=\sum^\ell_{j=1}\frac{ \sum^{r}_{i=1}p_{i}{n\choose m_j-1}\alpha^{n-i-m_j+1}_j}{\Pi^\ell_{k=1, k\not= j}(\alpha_j-\alpha_k)}\\
&=&\sum^\ell_{j=1}\frac{ {n\choose m_j-1}\alpha^{-m_j+1}_j\sum^{r}_{i=1}p_{i}\alpha^{n-i}_j}{\Pi^\ell_{k=1, k\not= j}(\alpha_j-\alpha_k)}\\
&=&\sum^\ell_{j=1}\frac{ {n\choose m_j-1}\alpha^{n-m_j+1}_j}{\Pi^\ell_{k=1, k\not= j}(\alpha_j-\alpha_k)}=\tilde F^{( r)}_{n}.
\ens
Therefore, sequence $\{ \tilde F^{( r)}_{n}\}$ satisfies linear recurrence relation (\ref{eq:1}) and is the IRS with respect to $\{p_j\}$.
\end{proof}
\eop

The IRS of a set of linear recurring sequences is a kind of basis so that every sequence in the set can be represented in terms of the IRS. For this purpose, we need to extend the indices of IRS $\tilde F^{(r)}_n$ of $A_r$ to the negative indices till $n=-r+1$, where $A_{r}$ be the set of all linear recurring sequences defined by the homogeneous linear recurrence relation (\ref{eq:1}) with coefficient set $E_{r}=\{ p_{1}, p_{2}, \ldots, p_r\}$, $p_r\not= 0$. For instance, $\tilde F^{(r)}_{-1}$ is defined by 

\[
\tilde F^{(r)}_{r-1}=p_1 \tilde F^{(r)}_{r-2} +p_2\tilde F^{(r)}_{r-3}+\cdots +p_r\tilde F^{(r)}_{-1},
\]
which implies $\tilde F^{(r)}_{-1}=1/p_r$ because $\tilde F^{(r )}_0=\cdots =\tilde F^{(r)}_{r-2}=0$ and $\tilde F^{(r)}_{r-1}=1$. Similarly, we may use 

\[
\tilde F^{(r)}_{r+j}=p_1\tilde F^{(r)}_{r+j-1}+p_2\tilde F^{(r)}_{r+j-2}+\cdots +p_r\tilde F^{(r)}_{j}
\]
to define $\tilde F^{(r)}_j$, $j=-2,\ldots, -r+1$ successively. The following theorem shows how to represent linear recurring sequences in terms of the extended IRS $\{ \tilde F^{(r)}_n\}_{n\geq -r+1}$. 

\begin{theorem}\label{thm:1.4}
Let $A_{r}$ be the set of all linear recurring sequences defined by the homogeneous linear recurrence relation (\ref{eq:1}) with coefficient set $E_{r}=\{ p_{1}, p_{2}, \ldots, p_r\}$ ($p_{r}\not= 0$), and let $\{\tilde F^{( r)}_{n}\}_{n\geq -r+1}$ be the (extended) IRS of $A_{r}$. Then, for any $a_n\in A_r$, there holds the expression of $a_n$ as 

\be\label{eq:5.8}
a_n=a_{r-1}\tilde F^{(r)}_n+\sum^{r-2}_{j=0} \sum^{r-2}_{k=j} a_kp_{r+j-k}\tilde F^{(r)}_{n-1-j}.
\ee
\end{theorem}

\begin{proof}
Considering the sequences on the both sides of (\ref{eq:5.8}), we find that both satisfy the same recurrence relation (\ref{eq:1}). Hence, to prove the equivalence of two sequences, we only need to show that they have the same initials, or have the same initial vector $(a_0, a_1, \ldots, a_{r-1})$ because of 
the uniqueness of the linear recurring sequence defined by (\ref{eq:1}) with fixed initials. First, for $n=r-1$, the conditions $\tilde F^{(r )}_0=\cdots =\tilde F^{(r)}_{r-2}=0$ and $\tilde F^{(r)}_{r-1}=1$ are applied on the right-hand side of (\ref{eq:5.8}) to obtain its value as $a_{r-1}$, which shows (\ref{eq:5.8}) holds for $n=r-1$. Secondly, for all $0\leq n\leq r-2$, one may write the right-hand side RHS of (\ref{eq:5.8}) as 

\[
RHS=0+\sum^{r-2}_{k=0} a_k \left( \sum^{k}_{j=0} p_{r+j-k}\tilde F^{(r)}_{n-1-j}\right)
=\sum^{r-2}_{k=0}a_{k}\delta_{k,n}=a_{n},
\]
where the Kronecker delta symbol $\delta_{k,n}$ equals to $1$ when $k=n$ and zero otherwise, and the following formula is used:

\be\label{eq:5.8.2}
\sum^{k}_{j=0} p_{r+j-k}\tilde F^{( r)}_{n-1-j}=\delta_{k,n}.
\ee
(\ref{eq:5.8.2}) can be proved by splitting it into three cases: $n=, >,$ and $< k$, respectively. If $n=k$, we have 

\bns
&&\sum^{k}_{j=0} p_{r+j-k}\tilde F^{( r)}_{n-1-j}=\sum^{k}_{j=0} p_{r+j-k}\tilde F^{( r)}_{k-1-j}\\
&=&p_{r}\tilde F^{( r)}_{-1}+p_{r-1}\tilde F^{( r)}_{0}+p_{r-2}\tilde F^{( r)}_{1}+\cdots +p_{r-k}\tilde F^{( r)}_{k-1}\\
&=&p_{r}\tilde F^{( r)}_{-1}=1,
\ens
where we apply use $\tilde F^{( r)}_{-1}=1/p_{r}$ and the fact $\tilde F^{( r)}_{j}=0$ for all $j=0,1,\ldots, k-1$ due to $k-1\leq r-3$ and the initial values of $\tilde F^{( r)}_{n}$ being zero for all $0\leq n\leq r-2$. 

If $n>k$, we have 

\[
\sum^{k}_{j=0} p_{r+j-k}\tilde F^{( r)}_{n-1-j}=p_{r}\tilde F^{( r)}_{n-k-1}+p_{r-1}\tilde F^{( r)}_{n-k}+\cdots+p_{r-k}\tilde F^{( r)}_{n-1}=0
\]
because $\tilde F^{( r)}_{j}=0$ for all $j=n-k-1, n-k, \ldots, n-1$ due to $-1<n-k-1\leq j \leq n-1\leq r-3$ and the definition of $\tilde F^{( r)}_{n}$.

Finally, for $0\leq n<k$, we may find 

\bns
&&\sum^{k}_{j=0} p_{r+j-k}\tilde F^{( r)}_{n-1-j}=p_{r}\tilde F^{( r)}_{n-k-1}+p_{r-1}\tilde F^{( r)}_{n-k}+\cdots+p_{r-k}\tilde F^{( r)}_{n-1}\\
&=&p_{r}\tilde F^{( r)}_{n-k-1}
+\cdots+p_{r-k}\tilde F^{( r)}_{n-1}+p_{r-k-1}\tilde F^{( r)}_{n}+p_{r-k-2}\tilde F^{( r)}_{n+1}
+\cdots +p_{1}\tilde F^{( r)}_{n+r-k-2}\\
&=&\tilde F^{( r)}_{n+r-k-1}=0,
\ens
where the inserted $r-k-1$ terms $\tilde F^{( r)}_{j}$, $j=n,n+1,\ldots n+r-k-2$, are zero due to the definition of $\tilde F^{( r)}_{n}$ and $0\leq n \leq j\leq  n+r-k-2 <r-2$. The last two steps of the above equations are from the recurrence relation (\ref{eq:1}) and the assumptions of $0\leq k\leq r-2$ and $1\leq n+1\leq n+r-k-1< r-1$, respectively. This completes the proof of theorem. 
\end{proof}
\eop

From Theorems \ref{thm:1.3} and \ref{thm:1.4}, we immediately obtain  

\begin{corollary}\label{cor:1.5}
Let $A_{r}$ be the set of all linear recurring sequences defined by the homogeneous linear recurrence relation (\ref{eq:1}) with coefficient set $E_{r}=\{ p_{1}, p_{2}, \ldots, p_r\}$ ($p_{r}\not= 0$), and let $\alpha_1,\alpha_2,\ldots, \alpha_\ell$ be the distinct real or complex roots of characteristic polynomial,  $p(t)=t^r-p_1t^{r-1}-\cdots -p_r$, of the recurrence relation (\ref{eq:1}) with the multiplicities $m_1,m_2,\ldots, m_\ell$, respectively, and $m_1+m_2+\cdots +m_\ell=r$. Then, for any $\{a_{n}\} \in A_{r}$, there holds 

\be\label{eq:5.8.3}
a_n=a_{r-1}\sum^\ell_{j=1}\frac{ {n\choose m_j-1}\alpha^{n-m_j+1}_j}{\Pi^\ell_{k=1, k\not= j}(\alpha_j-\alpha_k)}
+\sum^{r-2}_{j=0} \sum^{r-2}_{k=j} \sum^\ell_{j=1}a_kp_{r+j-k}\frac{ {n-j-1\choose m_j-1}\alpha^{n-j-m_j}_j}{\Pi^\ell_{k=1, k\not= j}(\alpha_j-\alpha_k)}
\ee
\end{corollary}

We now establish the relationship between a sequence in $A_r$ and the IRS of $A_r$ by using Toeplitz matrices.  A Toeplitz matrix may be defined as a $n\times n$ matrix $A$ where its entries $a_{i,j} = c_{i-j}$, for constants $c_{1-n}$, $\ldots, c_{n-1}$. A $m\times n$ block Toeplitz matrix is a matrix that can be partitioned into $m\times n$ blocks and every block  is a Toeplotz matrix. 

\begin{theorem}\label{thm:1.6}
Let $A_{r}$ be the set of all linear recurring sequences defined by the homogeneous linear recurrence relation (\ref{eq:1}) with coefficient set $E_{r}=\{ p_{1}, p_{2}, \ldots, p_r\}$, and let $\tilde F^{( r)}_{n}$ be the IRS of $A_{r}$. Then, 
for even $r$, we may write 

\be\label{eq:5.9}
\tilde F^{(r)}_n=\sum^{r-1}_{j=r/2}c_j a_{n+r-j} +\sum^{(3r/2)-1}_{j=r}c_j a_{n+r-j-1},
\ee
where coefficients $c_j$ ($r/2\leq j\leq (3r/2)-1$) satisfy 

\bn\label{eq:5.10}
&&A_{e}{\bf c}\nonumber\\
&&:=\left[ \begin{array} {llllllll} a_{r/2}& a_{(r/2)-1}&\cdots &a_1& a_{-1}& a_{-2}&\cdots &a_{-r/2}\\
a_{(r/2)+1}& a_{r/2}&\cdots &a_2& a_0& a_{-1}&\cdots &a_{-(r/2)+1}\\
\vdots& \vdots& \cdots&\vdots&\vdots& \vdots&\cdots&\vdots\\
a_{r-1}& a_{r-2}&\cdots&a_{r/2}&a_{(r/2)-2}&a_{(r/2)-3}&\cdots& a_{-1}\\
a_{r}&a_{r-1}&\cdots& a_{(r/2)+1}&a_{(r/2)-1}&a_{(r/2)-2}&\cdots &a_0\\
\vdots& \vdots& \cdots&\vdots&\vdots& \vdots&\cdots&\vdots\\
a_{(3r/2)-1}&a_{(3r/2)-2}&\cdots &a_{r}& a_{r-2}& a_{r-3}&\cdots&a_{(r/2)-1}
\end{array}\right]{\bf c} ={\bf e_r}\nonumber\\
\en
with ${\bf c}=(c_{r/2}, c_{(r/2)-1}, \ldots, c_r, c_{r+1}, c_{r+2},\ldots, c_{(3r/2)-1})^T$ and 
${\bf e_r}=(0,0,\ldots, 0,1)^T\in {\bR}^r$. For odd $r$, there holds 

\be\label{eq:5.11}
\tilde F^{(r)}_n=\sum^{3(r-1)/2}_{j=(r-1)/2}c_j a_{n+r-j-1},
\ee
where $c_j$ ($(r-1)/2\leq j\leq 3(r-1)/2$) satisfying 

\bn\label{eq:5.12}
&&A_{o}{\bf c}\nonumber\\
&:=&\left[ \begin{array} {llllllll} a_{(r-1)/2}& a_{((r-1)/2)-1}&\cdots 
&a_1&a_0&a_{-1}&\cdots &a_{-(r-1)/2}\\
a_{((r-1)/2)+1} &a_{(r-1)/2}& \cdots &a_2&a_1&a_0&\cdots&a_{-(r-1)/2+1}\\
\vdots&\vdots& \cdots& \vdots& \vdots& \vdots& \cdots&\vdots\\
a_{(3(r-1)/2)}&a_{(3(r-1)/2)-1}&\cdots &a_r&a_{r-1}&a_{r-2}&\cdots& a_{(r-1)/2}
\end{array}\right]{\bf c} ={\bf e_r}\nonumber\\
&&
\en
with ${\bf c}=(c_{(r-1)/2}, c_{((r-1)/2)-1}, \ldots, c_{r-1}, c_{r}, c_{r+1},\ldots, c_{(3(r-1)/2)})^T$ and ${\bf e_r}=(0,0,\ldots, 0,1)^T\in {\bR}^r$. If the $2\times 2$ block Toeplitz matrix $A_{e}$ and Toeplitz matrix $A_{o}$ are invertible, then we have unique expressions (\ref{eq:5.9}) and (\ref{eq:5.11}), respectively. 
\end{theorem}

\begin{proof} 
Since the operator $L: {\bZ}\times{\bZ}\mapsto {\bZ}$, $L(a_{n-1}, a_{n-2}, \ldots, a_{n-r}):=p_{1}a_{n-1}+p_{2}a_{n-2}+\cdots +p_ra_{n-r}=a_{n}$, is linear, sequence $\{a_{n}\}$ is uniquely determined by $L$ from a  given initial vector ($a_{0}$, $a_{1}, \ldots, a_{r-1}$). We may extend $\{ a_n\}_{n\geq 0}$ to a sequence with negative indices using the technique we applied before Theorem \ref{thm:1.4}. For instance, by defining $a_{-1}$ from $a_{r-1}=p_r a_{-1}+p_{r-1} a_0+p_{r-2} a_1+\cdots +p_1 a_{r-2})$, we obtain

\[
a_{-1}=\frac{1}{p_r}(a_{r-1}-p_{r-1} a_0-p_{r-2} a_1-\cdots -p_1 a_{r-2}), \quad p_r\not= 0. 
\]
Thus, the initial vector $(a_{-1}, a_{0}, \ldots, a_{r-2})$ generates $\{ a_n\}_{n\geq -1}$ or $\{a_{n-1}\}_{n\geq 0}$ by using the operator $L$. 
Since both sequences $\{a_n\}$ and $\{\tilde F^{(r )}_n\}$ satisfy the recurrence relation (\ref{eq:1}), they are generated by the same operator $L$. Hence, for even $r$, we may write $\tilde F^{(r)}_n$ as a linear combination of $a_{n+r-j}$ ($r/2\leq j\leq r-1$) and $a_{n+r-j-1}$ ($r\leq j\leq (3r/2)-1$) shown in (\ref{eq:5.9}) provided the initials of the sequences on the two sides are the same. Here, we mention again that the initials with negative indices have been defined by using (\ref{eq:1}).  Therefore, we substitute the initials of $\{ a_n\}$ with their negative extensions to the right-hand side of (\ref{eq:5.9}) and enforce the resulting linear combination equal to the initials of the left-hand side of (\ref{eq:5.9}), i.e., ${\bf e_r}$, the $r$th unit vector of ${\bR}^r$, which derives (\ref{eq:5.10}). Under the assumption of the invertibility of $A_{e}$, Toeplitz system (\ref{eq:5.10}) has a unique solution of ${\bf c}=(c_{r/2}, c_{(r/2)-1}, \ldots, c_r, c_{r+1}, c_{r+2},\ldots, $ $c_{(3r/2)-1})^T$, which shows the uniqueness of the expression of (\ref{eq:5.9}). (\ref{eq:5.11}) can be proved similarly. 
\end{proof}
\eop

\noindent{\bf Remark 2.1} It can be seen that not all Toeplitz matrix generated from the linear recurring sequence (\ref{eq:1}) is invertible. For example, consider
the sequence $\{ a_{n}\}$ generated by the linear recurrence relation $a_{n}=a_{n-1}+3a_{n-2}+a_{n-3}$ with $a_{0}=1$, $a_{1}=0$, $a_{2}=1$, there hold $a_{3}=a_{0}+3a_{1}+a_{2}=2$ and $a_{-1}=a_{2}-a_{1}-3a_{0}=-2$. From (\ref{eq:5.11}) 

\[
A_{o}=\left[ \begin{array}{lll}0 &1&2\\1&0&1\\-2&1&0\end{array}\right]
\]
which is not invertible. 

If $p_j=1$ ($1\leq j\leq r$), the IRS is the high order Fibonacci sequence. For instance, $\tilde F^{(3)}_n$, $\tilde F^{(4)}_n$, $\tilde F^{(5)}_n$, $\tilde F^{(6)}_n$, $\tilde F^{(7)}_n$, etc. are Tribonacci numbers (\seqnum{ A000073}), Tetranacci numbers (\seqnum{ A000078}), Pentanacci numbers (\seqnum{ A001591}), Heptanacci numbers (\seqnum{ A122189}), Octanacci numbers (\seqnum{ A079262}), etc., respectively. 

\section{Examples with respect to $E_2$ and $E_3$}

We now give some examples of IRS for particular $r$ starting from $r=2$. Let $A_{2}$ be the set of all linear recurring sequences defined by the homogeneous linear recurrence relation (\ref{eq:5.1}) with coefficient set $E_{2}=\{ p_{1}, p_{2}\}$. Hence, from Theorems \ref{thm:1.3} and \ref{thm:1.4} and the definition of the impulse response sequence of $A_{2}$ with respect to $E_{2}$, we obtain 

\begin{corollary}\label{cor:5.2}
Let $A_{2}$ be the set of all linear recurring sequences defined by the homogeneous linear recurrence relation (\ref{eq:5.1}) with coefficient set $E_{2}=\{ p_{1}, p_{2}\}$, and let $\{\tilde F^{(2)}_{n}\}$ be the IRS 
of $A_{2}$. Suppose $\alpha$ and $\beta$ are two roots of the characteristic polynomial of $A_{2}$, which do not need to be distinct. Then 

\be\label{eq:5.2}
\tilde F^{(2)}_{n}=\left\{ \begin{array}{ll} \frac{\alpha^{n}-\beta^{n}}{\alpha -\beta}, & if\,\, \alpha\not= \beta;\\
n \alpha^{n-1}, &if\,\, \alpha =\beta.\end{array}\right.
\ee
In addition, every $\{ a_{n}\}\in A_{2}$ can be written as 

\be\label{eq:5.3}
a_{n}=a_{1}\tilde F^{(2)}_{n}-\alpha \beta a_{0}\tilde F^{(2)}_{n-1},
\ee
and $a_{n}$ reduces to $a_{1}\tilde F^{(2)}_{n}-\alpha^{2}a_{0}\tilde F^{(2)}_{n-1}$ when $\alpha=\beta$.

Conversely, there holds a expression of $\tilde F^{(2)}_n$ in terms of $\{a_n\}$ as 

\be\label{eq:5.4}
\tilde F^{(2)}_n= c_1 a_{n+1}+c_2 a_{n-1},
\ee
where 

\be\label{eq:5.5}
c_1= \frac{a_{1}-a_{0}p_{1}}{p_{1}(a^{2}_{1}-a_{0}a_{1}p_{1}-a^{2}_{0}p_{2})},\quad c_2= -\frac{a_{1}p_2}{p_{1}(a^{2}_{1}-a_{0}a_{1}p_{1}-a^{2}_{0}p_{2})},
\ee
provided that $p_1\not= 0$, and $a^{2}_{1}-a_{0}a_{1}p_{1}-a^{2}_{0}p_{2}\not= 0$. 
\end{corollary}

\begin{proof}
(\ref{eq:5.2}) is a special case of (\ref{eq:5.7}) and (\ref{eq:5.7.2}) and can be found from (\ref{eq:5.7}) and (\ref{eq:5.7.2}) by using the substitution $r=2$, $a_0=0$ and $a_1=1$. Again from (\ref{eq:5.8}) and (\ref{eq:5.8.3})  

\bn\label{eq:5.3-0}
a_n&=&a_{1}\tilde F_{n}+ a_{0}p_{2}\tilde F_{n-1}=a_{1}\tilde F_{n}-\alpha \beta a_{0}\tilde F_{n-1}\nonumber\\
&=&\left\{ \begin{array}{ll} a_1 \frac{\alpha^{n}-\beta^{n}}{\alpha -\beta} -\alpha \beta a_0\frac{\alpha^{n-1}-\beta^{n-1}}{\alpha -\beta}, & if\,\, \alpha\not= \beta;\\
a_1(n\alpha^{n-1})-\alpha^2 ((n-1)\alpha^{n-2}), &if \,\, \alpha =\beta \end{array}.\right.\\
\en

From (\ref{eq:5.9}) and (\ref{eq:5.10}), we now prove (\ref{eq:5.4}) and (\ref{eq:5.5}), respectively. Denote by $L: {\bZ}\times{\bZ}\mapsto {\bZ}$ the operator $L(a_{n-1}, a_{n-2}):=p_{1}a_{n-1}+p_{2}a_{n-2}=a_{n}$. As what we presented before that $L$ is linear, and the  sequence $\{a_{n}\}$ is uniquely determined by $L$ from a given initial vector ($a_{0}$, $a_{1}$). Define $a_{-1}=(a_{1}-p_{1}a_{0})/p_{2}$, then $(a_{-1}, a_{0})$ is the initial vector that generates $\{ a_{n-1}\}_{n\geq 0}$ by $L$. Similarly, the vector $(a_{1}, p_{1}a_{1}+p_{2}a_{0})$ generates sequence $\{ a_{n+1}\}_{n\geq }$ by using $L$. Note the initial vectors of $\tilde F^{(2)}_{n}$ is $(0,1)$. Thus (\ref{eq:5.4}) holds if and only if the initial vectors on the two sides are equal:

\be\label{eq:5.6}
(0,1)=c_1 (a_{1}, p_{1}a_{1}+p_{2}a_{0})+c_2\left( \frac{a_{1}-p_{1}a_{0}}{p_{2}}, a_{0}\right),
\ee
which yields the solutions (\ref{eq:5.5}) for $c_1$ and $c_2$ and completes the proof of the corollary.
\end{proof}
\eop

\medbreak
\noindent{\bf Remark 3.1} Recall that \cite{HS} 
presented the following result. 

\begin{proposition}\label{pro:5.1}\cite{HS}
Let $\{ a_n\}$ be a sequence of order $2$ satisfying linear recurrence relation (\ref{eq:5.1}), and let $\alpha$ and $\beta$ be two roots of the characteristic polynomial $x^{2}-p_{1} x-p_{2}=0$ of the relation (\ref{eq:5.1}). Then 

\be\label{eq:5.1-2}
a_n=\left\{ \begin{array}{ll} \left( \frac{a_1-\beta a_0}{\alpha-\beta}\right) \alpha^n- \left(\frac{a_1-\alpha a_0}{\alpha-\beta}\right) \beta^n, & if\,\, \alpha\not= \beta;\\
na_1 \alpha^{n-1}-(n-1)a_0\alpha^{n}, &if\,\, \alpha =\beta.\end{array}\right.
\ee
\end{proposition}
\cite{HS} also shows a method of finding the expressions of the linear recurring sequences of order $2$ and the interrelationship among those sequences. However, \cite{HS} also pointed out ``the method presented in this paper (i.e., Proposition \ref{pro:5.1}) cannot be extended to the higher order setting.'' In Section $2$, we have shown our method based on the IRS can be extended to the higher setting with Proposition \ref{pro:5.1} as a particular case of our method. In addition, (\ref{eq:5.1-2}) can be derived from (\ref{eq:5.3-0}). 

Corollary \ref{cor:5.2} presents the interrelationship between a linear recurring sequence with respect to $E_{2}=\{ p_{1}, p_{2})$ and its IRS, which can be used to establish the identities of one sequence from the identities of other sequences. 

\medbreak
\noindent{\bf Example 3.1} Let us consider $A_{2}$, the set of all linear recurring sequences defined by the homogeneous linear recurrence relation (\ref{eq:5.1}) with coefficient set $E_{2}=\{ p_{1}, p_{2}\}$. If $E_{2}=\{1,1\}$, then the corresponding characteristic polynomial has roots $\alpha =(1+\sqrt{5})/2$ and $\beta =(1-\sqrt{5})/2$, and (\ref{eq:5.4}) gives the expression of the ISR of $A_{2}$, which is Fibonacci sequence $\{ F_{n}\}$:  

\[
F_{n}= \frac{1}{\sqrt{5}}\left\{ \left( \frac{1+\sqrt{5}}{2}\right)^{n}-\left( \frac{1-\sqrt{5}}{2}\right)^{n}\right\}.
\]
The sequence in $A_{2}$ with the initial vector $(2,1)$ is Lucas sequence $\{ L_{n}\}$. From (\ref{eq:5.3}) and (\ref{eq:5.4}) and noting $\alpha \beta=-1$, we have the well-known formulas (see, for example, \cite{Koshy}):

\[
L_{n}=F_{n}+2F_{n-1}=F_{n+1}+F_{n-1}, \quad F_{n}=\frac{1}{5}L_{n-1}+\frac{1}{5}L_{n+1}.
\]
By using the above formulas, one may transfer identities of Fibonacci number sequence to those of Lucas number sequence and vice verse. For instance, the above relationship can be used to prove that the following two identities are equivalent:  

\bns
&&F_{n+1}F_{n+2}-F_{n-1}F_{n}=F_{2n+1}\\
&&L^{2}_{n+1}+L^{2}_{n}=L_{2n}+L_{2n+2}.
\ens
It is clear that both of the identities are equivalent to the Carlitz identity, $F_{n+1}L_{n+2}-F_{n+2}L_{n}=F_{2n+1}$, shown in \cite{Car}.
\medbreak

\noindent{\bf Example 3.2} Let us consider $A_{2}$, the set of all linear recurring sequences defined by the homogeneous linear recurrence relation (\ref{eq:5.1}) with coefficient set $E_{2}=\{ p_1=p, p_2=1\}$. Then (\ref{eq:5.1-2}) tell us that $\{ a_{n}\}\in A_{2}$ satisfies 

\be\label{eq:9**}
a_n=\frac{2a_1-(p-\sqrt{4+p^2})a_0}{2\sqrt{4+p^2}} \alpha^n-\frac{2a_1-(p+\sqrt{4+p^2})a_0}{2\sqrt{4+p^2}}\left( -\frac{1}{\alpha}\right)^n,
\ee
where $\alpha$ is defined by 

\be\label{eq:9*}
\alpha =\frac{p+\sqrt{4+p^2}}{2} \,\,\,\mbox{and} \,\, \, \beta =-\frac{1}{\alpha}=\frac{p-\sqrt{4+p^2}}{2}.
\ee
Similarly, let $E_{2}=\{1,q\}$. Then 

\[
a_n=\left\{ \begin{array}{ll} \frac{2a_1-(1-\sqrt{1+4q})a_0}{2\sqrt{1+4q}}\alpha_1^n-
\frac{2a_1-(1+\sqrt{1+4q})a_0}{2\sqrt{1+4q}}\alpha_2^n, & if\,\, q\not=- \frac{1}{4};\\
\frac{1}{2^n} (2n a_1-(n-1) a_0), &if \,\, q=-\frac{1}{4},\end{array}\right.
\]
where $\alpha=\frac{1}{2}(1+\sqrt{1+4q})$ and $\beta=\frac{1}{2}(1-\sqrt{1+4q})$ are solutions of equation $x^2-x-q=0$. The first special case (\ref{eq:9**}) was studied by Falbo in \cite{Fal}. If $p=1$, the sequence is clearly the Fibonacci sequence. If $p=2$ ($q=1$), the corresponding sequence 
is the sequence of numerators (when two initial conditions are $1$ and $3$) or denominators (when two initial conditions are $1$ and $2$) of the convergent of a continued fraction to $\sqrt{2}$: 
$\{ \frac{1}{1}$, $\frac{3}{2}$, $\frac{7}{5},$ $\frac{17}{12}, \frac{41}{29} \ldots\}$, called the closest rational approximation sequence to $\sqrt{2}$. The second special case is for the case of $q=2$ ($p=1$), the resulting $\{ a_n\}$ is the Jacobsthal type sequences  (See Bergum, Bennett, Horadam, and Moore \cite{BBHM}). 

From Corollary \ref{cor:5.2}, for $E_{2}=\{p,1\}$, the IRS of $A_2$ with respect to $E_2$ is 

\[
\tilde F^{(2)}_{n}=\frac{1}{\sqrt{4+p^{4}}}\left\{ \left( \frac{p+\sqrt{4+p^{2}}}{2}\right)^{n}-\left( \frac{p-\sqrt{4+p^{2}}}{2}\right)^{n}\right\}.
\]
In particular, the IRS for $E_{2}=\{2,1\}$ is the well-known Pell number sequence $\{P_{n}\}=\{0,1,2,5,12,29,\ldots\}$ with the expression 

\[
P_{n}=\frac{1}{2\sqrt{2}}\left\{ (1+\sqrt{2})^{n}-(1-\sqrt{2})^{n}\right\}.
\]
Similarly, for $E_{2}=\{1,q\}$, the IRS of $A_2$ with respect to $E_2$ is 

\[
\tilde F^{(2)}_{n}=\frac{1}{\sqrt{1+4q}}\left\{ \left( \frac{1+\sqrt{1+4q}}{2}\right)^{n}-\left( \frac{1-\sqrt{1+4q}}{2}\right)^{n}\right\}.
\]
In particular, the ISR for $E_{2}=\{1,2\}$ is the well-known Jacobsthal number sequence $\{J_{n}\}=\{0,1,1,3,5,11,21,\ldots\}$ with the expression

\[
J_{n}=\frac{1}{3}\left( 2^{n}-(-1)^{n}\right). 
\]
The Jacobsthal-Lucas number $\{ j_{n}\}$ in $A_{2}$ with respect to $E_{2}=\{1,2\}$ satisfying $j_{0}=2$ and $j_{1}=1$ has the first few elements as $\{ 2,1,5,7,17,31,\ldots\}$. From (\ref{eq:5.3}), one may have 

\[
j_{n}=J_{n}+4J_{n-1}=2^{n}+(-1)^{n}.
\]
In addition, the above formula can transform all identities of Jacobsthal-Lucas number sequence to those of Jacobsthal number sequence and vice versa.  For example, we have 

\bns
&&J^{2}_{n}+4J_{n-1}J_{n}=J_{2n},\\
&&J_{m}J_{n-1}-J_{n}J_{m-1}=(-1)^{n}2^{n-1}J_{m-n},\\
&&J_{m}J_{n}+2J_{m}J_{n-1}+2J_{n}J_{m-1}=J_{m+n}
\ens
from 

\bns
&&j_{n}J_{n}=J_{2n},\\
&&J_{m}j_{n}-J_{n}j_{m}=(-1)^{n}2^{n+1}J_{m-n},\\
&&J_{m}j_{n}-J_{n}j_{m}=2J_{m+n},
\ens
respectively. Similarly, we can show that the following two identities are equivalent:

\[
j_{n}=J_{n+1}+2J_{n-1},\quad J_{n+1}=J_{n}+2J_{n-1}.
\]

\medbreak
\noindent{\bf Remark 3.2} Corollary \ref{cor:5.2} can be extended to the linear nonhomogeneous recurrence relations of order $2$ with the form: $a_{n}=p a_{n-1}+q a_{n-2} +\ell$ for $p+q\not= 1$. It can be seen that the above recurrence relation is equivalent to the homogeneous form (\ref{eq:5.1}) $b_n=p b_{n-1}+q b_{n-2}$, where $b_n=a_n-k$ and $k=\frac{\ell}{1-p-q}.$

\medbreak
\noindent{\bf Example 3.3}
An obvious example of Remark 3.2  is the Mersenne number $M_n=2^n-1$ ($n\geq 0$), which satisfies the linear recurrence relation of order $2$: $M_n=3 M_{n-1}-2 M_{n-2}$ ( with $ M_0=0$ and $M_1=1$) and the non-homogeneous recurrence relation of order $1$: $M_n=2M_{n-1}+1$ (with $M_0=0$).  It is easy to check that sequence $M_n=(k^{n}-1)/(k-1)$ satisfies both the homogeneous recurrence relation of order $2$, $M_n=(k+1)M_{n-1}-k M_{n-2}$, and the non-homogeneous recurrence relation of order $1$, $M_n=k M_{n-1}+1$, where $M_0=0$ and $M_1=1$. Here, $M_n$ is the IRS with respect to $E_2=\{3,-2\}$. Another example is Pell number sequence that satisfies both 
homogeneous recurrence relation $P_n=2P_{n-1}+P_{n-2}$ and the non-homogeneous relation 
$\bar P_n=2\bar P_{n-1}+\bar P_{n-2}+1$, where $P_n=\bar P_n+1/2$.  

\medbreak
\noindent{\bf Remark 3.3} 
In \cite{NZM}, Niven, Zuckerman, and Montgomery studied some properties of $\{G_n\}_{n\geq 0}$ and $\{ H_n\}_{n\geq 0}$ 
defined  respectively by the linear recurrence relations of order $2$:
\[
G_n=pG_{n-1}+q G_{n-2} \quad and \quad H_n=p H_{n-1}+q H_{n-2}
\]
with initial conditions $G_0=0$ and $G_1=1$ and $H_0=2$ and $H_1=p$, respectively. Clearly, $G_n=\tilde F^{(2)}_n$, the IRS of $A_2$ with respect to $E_2=\{p_1=p, p_2=q\}$. Using Corollary \ref{cor:5.2}, we may rebuild the relationship between the sequences $\{G_n\}$ and $\{ H_n\}$:

\bns
&&H_{n}=pG_{n}+2qG_{n-1},\\
&&G_{n}=\frac{q}{p^{2}+4q}H_{n-1}+\frac{1}{p^{2}+4q}H_{n+1}.
\ens

\medbreak

We now give more examples of higher order IRS in $A_3$. 
\medbreak

\noindent{\bf Example 3.4} Consider set $A_3$ of all linear recurring sequences defined by (\ref{eq:1}) with coefficient set $E_3=\{ p_1=1, p_2=1, p_3=1\}$. The IRS of $A_3$ is the tribonacci number sequence $\{ 0,0,1,1,2,4,7,13,24,44,\ldots\}$ (A000073). From  Proposition \ref{pro:3.1},  the generating  function of the tribonacci sequence is  

\[
\tilde P_3(t)=\frac{t^2}{1-t-t^2-t^3}.
\]
Theorem \ref{thm:1.3} gives the expression of the general term of tribonacci sequence as 

\[
\tilde F^{(3)}_n=\frac{\alpha^n_1}{(\alpha_1-\alpha_2)(\alpha_1-\alpha_3)}+
\frac{\alpha^n_2}{(\alpha_2-\alpha_1)(\alpha_2-\alpha_3)}+\frac{\alpha^n_3}{(\alpha_3-\alpha_1)(\alpha_3-\alpha_2)},
\]
where 

\bns
&&\alpha_1=-\frac{1}{3}-\frac{C}{3}+\frac{2}{3C},\\
&&\alpha_2=-\frac{1}{3}+\frac{C}{6}(1+i\sqrt{3})-\frac{1}{3C}(1-i\sqrt{3}),\\
&&\alpha_3=-\frac{1}{3}+\frac{C}{6}(1-i\sqrt{3})-\frac{1}{3C}(1+i\sqrt{3}),
\ens
in which $C=((6\sqrt{33}-34)/2)^{1/3}.$ The tribonacci-like sequence $\{ a_n\}=\{ 2,1,1,4,6,11,$ $21,\ldots\}$ (A141036) is in the set $A_3$. From Theorem \ref{thm:1.4} and its corollary \ref{cor:1.5}, $a_n$ can be presented as 

\bn\label{eq:5.13}
a_n&=&a_2\tilde F^{(3)}_n+\sum^1_{j=0}\sum^1_{k=j}a_k\tilde F^{(3)}_{n-j-1}. \nonumber\\
&=&a_2\tilde F^{(3)}_n+(a_0+a_1)\tilde F^{(3)}_{n-1}+a_1\tilde F^{(3)}_{n-2}.
\en
Using (\ref{eq:5.11}) and (\ref{eq:5.12}) in Theorem \ref{thm:1.6}, we may obtain the expression 

\be\label{eq:5.14}
\tilde F^{(3)}_n=\frac{6}{19} a_{n+1}-\frac{4}{19}a_n-\frac{1}{19} a_{n-1}, \quad n\geq 0,
\ee
where $a_{-1}:=-2$ due to the linear recurrence relation $a_n=a_{n-1}+a_{n-2}+a_{n-3}$. 

It is easy to see there holds identity 

\[
\tilde F^{(3)}_n=2 \tilde F^{(3)}_{n-1}-\tilde F^{(3)}_{n-4}, \quad n\geq 4,
\]
because the right-hand side is equal to $\tilde F^{(3)}_{n-1}+\tilde F^{(3)}_{n-2}+\tilde F^{(3)}_{n-3}$ after substituting into $\tilde F^{(3)}_{n-1}=\tilde F^{(3)}_{n-2}+\tilde F^{(3)}_{n-3}+\tilde F^{(3)}_{n-4}$. Thus, by using (\ref{eq:5.14}) we have an identity for $\{ a_n\}$:

\bns
&&6 a_{n+1}-16 a_n+7 a_{n-1}+2a_{n-2}+6 a_{n-3}-4 a_{n-4}-a_{n-5}=0\\
\ens
for all $n\geq 5$. Similarly from the identity $a_n=2a_{n-1}-a_{n-4}$ ($n\geq 4$) and relation (\ref{eq:5.13}), there holds the identity 

\[
\tilde F^{(3)}_{n}+\tilde F^{(3)}_{n-1}-5\tilde F^{(3)}_{n-2}-2\tilde F^{(3)}_{n-3}+\tilde F^{(3)}_{n-4}+3\tilde F^{(3)}_{n-5}+\tilde F^{(3)}_{n-6}=0
\]
for all $n\geq 6$.

It is easy to see there exists $r$ number sequences, denoted by $\{a^{(j)}_n\}$ ($j=1,2,\ldots, r$), in $A_r$ with respect to $E_r$ such that for any $\{ a_n\}\in A_r$ there holds 

\[
a_n=\sum^r_{j=1} c_j a^{(j)}_n, \quad n\geq 0, 
\]
where ${\bf c}=(c_1,c_2,\ldots, c_r)$ can be found from the system consisting of the above equations for $n=0, 1,\ldots, r-1$. In this sense, we may call $\{a^{(j)}_n\}$ ($j=1,2,\ldots, r$) a basis of $A_r$ with respect to $E_r$. For instance, for $r=2$, let $\{ a^{(1)}_n\}$ be $\{F_n\}$, the IRS of $A_2$ with respect to $E_2=\{1,1\}$ (i.e., the Fibonacci sequence), and $\{ a^{(2)}_n\}$ be the Lucas number sequence $\{ L_n\}$. Then, $\{ \{ F_n\}, \{ L_n\}\}$ is a basis of $A_2$ because, for any $\{ a_n\}\in A_2$ with respect to $E_2=\{1,1\}$ and the initial vector $(a_0, a_1)$, there holds 

\[
a_n=\left(a_1-\frac{1}{2} a_0\right)F_n +\frac{1}{2} a_0 L_n,
\]
where the coefficients of the above linear combination are found from the system 

\[
\left[\begin{array}{ll}0&2\\1&1\end{array}\right] \left[ \begin{array}{l} c_1\\ c_2\end{array}\right]=\left[ \begin{array}{l}a_0\\a_1\end{array}\right].
\]

Obviously, if $\{ L_n\}$ is replaced by any sequence $\{ a^{(2)}_n\}$ in $A_2$ with respect to $E_2=\{1,1\}$, provided the initial vector $(a^{(2)}_0, a^{(2)}_1)$ of $\{ a^{(2)}_n\}$ satisfies 

\be\label{eq:5.15}
det \left[\begin{array}{ll}0&a^{(2)}_0\\1&a^{(2)}_1\end{array}\right] \not= 0,
\ee
then $\{ \{ F_n\}, \{ a^{(2)}_n\}\}$ is a basis of $A_2$ with respect to 
$E_2=\{ 1,1\}$. However, if $a^{(2)}_0=1$ and $a^{(2)}_1=0$, then the corresponding 
$a^{(2)}_n=F_{n-1}$, and the corresponding basis $\{ \{ F_n\}, \{ a^{(2)}_n=F_{n-1}\}\}$ is called a trivial basis. Thus, for $A_2$ with respect to $E_2=\{ p_1, p_2\}$, $\{ \{ \tilde F^{(2)}_n\}, \{ a^{(2)}_n\}\}$ forms a non-trivial basis, if the initial vector $(a^{(2)}_0, a^{(2)}_1)$ of $\{ a^{(2)}_n\}$ satisfies (\ref{eq:5.15}) and $a^{(2)}_0\not= 1/p_2$ when $a^{(2)}_1=0$. The second condition guarantees that the basis is not trivial, otherwise $a^{(2)}_n=\tilde F^{(2)}_{n-1}$.

\section{A type of Identities of IRS in $A_2$}

Let $A_{2}$ be the set of all linear recurring sequences defined by the homogeneous linear recurrence relation (\ref{eq:5.1}) with coefficient set $E_{2}=\{p_1=p, p_2=q\}$, and let $\tilde F^{(2)}$ be the IRS of $A_2$. Inspired by \cite{Hsu}, we give a nonlinear combinatorial expression involving $\tilde F^{(2)}$ and a numerous identities based on the expression. Using the interrelationship between the IRS and a linear recurring sequence in $A_2$, one may obtain many identities involving sequences in $A_2$. More precisely, let us consider the following extension of  the results in \cite{Hsu} for the Fibonacci numbers to the general number sequences in $A_2$. Suppose $\{ a_n\}_{n\in{\bN}}$ be a nonzero sequence defined by the recurrence relation 

\be\label{eq:0.1}
a_{n}=p_1a_{n-1}+p_2a_{n-2},\quad n\geq 2, p_1,p_2\not= 0, 
\ee
with the initial conditions $a_0=0$ and any nonzero $a_1$. Here, $a_1$ must be nonzero, otherwise $a_n\equiv 0$. Hence, we may normalize $a_1$ to be $a_1=1$ by define a new sequence $g_n=a_n/a_1$ satisfying the same recurrence relation (\ref{eq:0.1}). Thus, under the assumption, our sequence $\{ a_n\}$ is the IRS $\{ \tilde F^{(2)}_n\}$ of $A_2$ with respect to $E_2=\{ p_1, p_2\}$. We now give a nonlinear combinatorial expression involving $\tilde F^{(2)}_n$. Our result will extend to the case of $a_0\not= 0$ and $a_1=p_1a_0$ later. In addition, sequence 
$\{\tilde F^{(2)}_n\}_{n\in{\bN}}$ can be extended to the the case of $\{	\tilde F^{(2)}_r\}_{r\in{\bZ}}$ by using the same recurrence relation for $r\geq 1$ and $\tilde F^{(2)}_{r+1}=p_1\tilde F^{(2)}_{r}+p_2\tilde F^{(2)}_{r-1}$ while $r\leq -3$. 

\begin{lemma}\label{lem:0.1}
For any $m\in{\bN}$ and $r\in{\bZ}$ there holds 

\be\label{eq:0.2}
\tilde F^{(2)}_{m+r}=\tilde F^{(2)}_m \tilde F^{(2)}_{r+1}+p_2 \tilde F^{(2)}_{m-1} \tilde F^{(2)}_r.
\ee
\end{lemma}

\begin{proof}
For an arbitrarily $r\in{\bZ}$, we have 

\[
\tilde F^{(2)}_{r+1}=\tilde F^{(2)}_1 \tilde F^{(2)}_{r+1}+p_2\tilde F^{(2)}_0 \tilde F^{(2)}_r
\]
because $\tilde F^{(2)}_0=0$ and $\tilde F^{(2)}_1=1$. Assume (\ref{eq:0.2}) is true for $n\in{\bN}$, $n\geq 1$, and an arbitrary $r\in{\bZ}$, namely,

\[
\tilde F^{(2)}_{r+n}=\tilde F^{(2)}_n \tilde F^{(2)}_{r+1}+p_2\tilde F^{(2)}_{n-1} \tilde F^{(2)}_r, \quad r\in{\bZ}.
\]
Then,

\[
\tilde F^{(2)}_{r+n+1}=\tilde F^{(2)}_n \tilde F^{(2)}_{r+2}+p_2\tilde F^{(2)}_{n-1} \tilde F^{(2)}_{r+1}.
\]
On the hand, 

\bns
&&\tilde F^{(2)}_{n+1} \tilde F^{(2)}_{r+1}+p_2 \tilde F^{(2)}_{n}\tilde F^{(2)}_{r}=(p_1 \tilde F^{(2)}_n+p_2\tilde F^{(2)}_{n-1})\tilde F^{(2)}_{r+1}+p_2\tilde F^{(2)}_n \tilde F^{(2)}_{r}\\
&&=\tilde F^{(2)}_n \tilde F^{(2)}_{r+2}+p_2 \tilde F^{(2)}_{n-1} \tilde F^{(2)}_{r+1},
\ens
which implies 

\[
\tilde F^{(2)}_{r+n+1}=\tilde F^{(2)}_{n+1} \tilde F^{(2)}_{r+1}+p_2 \tilde F^{(2)}_{n}\tilde F^{(2)}_{r}
\]
and completes the proof with the mathematical indiction. 
\end{proof}
\eop

A direct proof of (\ref{eq:0.2}) can also be given. Actually, every $\tilde F^{(2)}_m\tilde F^{(2)}_{r+1}+p_2 \tilde F^{(2)}_{m-1} \tilde F^{(2)}_r$ can be reduced to $\tilde F^{(2)}_1 \tilde F^{(2)}_{r+m}+p_2\tilde F^{(2)}_0\tilde F^{(2)}_r=\tilde F^{(2)}_{r+m}$ by using the recurrence relation (\ref{eq:0.1}). 

\begin{theorem}\label{thm:0.2}
For any given $m,n\in{\bN}_0$ and $r\in{\bZ}$ there holds 

\be\label{eq:0.3}
\tilde F^{(2)}_{r+mn}=\sum^n_{j=0} {n\choose j} (\tilde F^{(2)}_m)^j(p_2	\tilde F^{(2)}_{m-1})^{n-j} \tilde F^{(2)}_{r+j}.
\ee
\end{theorem}

\begin{proof}
Let $F(t)=\tilde F^{(2)}_{r+mt}$. Then from Lemma \ref{lem:0.1}

\bns
&&\Delta F(t)=F(t+1)-F(t)=\tilde F^{(2)}_{r+mt+m}-\tilde F^{(2)}_{r+mt}\\
&=&\tilde F^{(2)}_m\tilde F^{(2)}_{r+mt+1}+(p_2\tilde F^{(2)}_{m-1}-1)\tilde F^{(2)}_{r+mt}.
\ens
Thus, there holds symbolically 

\[
(\Delta -(p_2\tilde F^{(2)}_{m-1}-1)I) \tilde F^{(2)}_{r+mt}=\tilde F^{(2)}_m\tilde F^{(2)}_{r+mt+1}.
\]
Using the operator $\Delta -(p_2\tilde F^{(2)}_{m-1}-1)I$ defined above $j$ times, we find 

\[
(\Delta -(p_2\tilde F^{(2)}_{m-1}-1)I)^j \tilde F^{(2)}_{r+mt}=(\tilde F^{(2)}_m)^j \tilde F^{(2)}_{r+mt+j}, \quad j\in{\bN}.
\]
Furthermore, noting the symbolic relation $E=I+\Delta$ and the last symbolical expression, one may find 

\bns
&& F(n)=\tilde F^{(2)}_{r+mn}=\left. E^n \tilde F^{(2)}_{r+mt}\right|_{t=0}=\left. (I+\Delta)^n \tilde F^{(2)}_{r+mt}\right|_{t=0}\\
&=& \left. (p_2\tilde F^{(2)}_{m-1}I+(\Delta -(p_2\tilde F^{(2)}_{m-1}-1)I)^n \tilde F^{(2)}_{r+mt}\right|_{t=0}\\
&=&\left. \sum^n_{j=0} {n\choose j} (p_2 \tilde F^{(2)}_{m-1})^{n-j} (\Delta -(p_2\tilde F^{(2)}_{m-1}-1)I)^j \tilde F^{(2)}_{r+mt}\right|_{t=0}\\
&=& \sum^n_{j=0} {n\choose j} (p_2 \tilde F^{(2)}_{m-1})^{n-j} (\tilde F^{(2)}_m)^j \tilde F^{(2)}_{r+j}
\ens
completing the proof of the theorem.
\end{proof}
\eop

\noindent{\bf Remark 3.1} The nonlinear expression for the case of $\{ a_n\}$ with $a_0=0$ and $a_1\not= 0$ can be extended to the case $a_0\not= 0$ and $a_1=p_1 a_0$. We may normalize $a_0=1$ and define $a_{-1}=0$ from the recurrence relation $a_1=p_1a_0+p_2 a_{-1}$. Hence, the sequence $\{ \hat F_n=a_{n-1}\}$ satisfies recurrence relation (\ref{eq:0.1}) for $n\geq 1$ with the initials $\hat F_0=0$ and $\hat F_1=1$. Hence, from (\ref{eq:0.3}) we have the nonlinear expression for $\hat F_n$ as 

\[
\hat F_{r+mn}=\sum^n_{j=0}{n\choose j} \hat F^k_{m-1}(q \hat F_{m-2})^{n-j} \hat F_{r+k}
\]
for $m\geq 1$ and $r\geq 0$. 

Similar to the last section and Remark 3.1, we may use the extension technique to define $\tilde F^{(2)}_n$ for negative integer index $n$. For example, substituting $n=1$ into (\ref{eq:0.1}) yields $\tilde F^{(2)}_1=p_1\tilde F^{(2)}_0+p_2\tilde F^{(2)}_{-1}$, which defines $\tilde F^{(2)}_{-1}=1/q$. With $r=-mn-1$, $r=-mn$, and $r=-mn+1$ in (\ref{eq:0.3}), a class of identities for $\tilde F^{(2)}_n$ with negative indices can be obtained as follows.

\begin{corollary}\label{cor:0.3}
For $m\geq 1$ and $n\geq 0$ there hold the identities

\bn\label{eq:0.4}
&&\sum^n_{j=0}p_2^{n-j+1}{n\choose j}(\tilde F^{(2)}_m)^j (\tilde F^{(2)}_{m-1})^{n-j}\tilde F^{(2)}_{j-mn-1}=1,\nonumber\\
&&\sum^n_{j=0}p_2^{n-j}{n\choose j}(\tilde F^{(2)}_m)^j (\tilde F^{(2)}_{m-1})^{n-j}\tilde F^{(2)}_{j-mn}=0, \nonumber\\
&&\sum^n_{j=0}p_2^{n-j}{n\choose j}(\tilde F^{(2)}_m)^j (\tilde F^{(2)}_{m-1})^{n-j}\tilde F^{(2)}_{j-mn+1}=1.
\en
\end{corollary}

Similarly, substituting $m=2,3,$ and $4$ into (\ref{eq:0.3}) and noting $\tilde F^{(2)}_2=p$, $\tilde F^{(2)}_3=p^2+q$, and $\tilde F^{(2)}_4=p(p^2+2q)$,  we have 

\begin{corollary}\label{cor:0.4}
For $n\geq 0$, there hold identities

\bn\label{eq:0.5}
&&\sum^n_{j=0}p_1^jp_2^{n-j}{n\choose j}\tilde F^{(2)}_{r+j}=\tilde F^{(2)}_{r+2n},\nonumber\\
&&\sum^n_{j=0}(p_1^2+p_2)^j(p_1p_2)^{n-j}{n\choose j}\tilde F^{(2)}_{r+j}=\tilde F^{(2)}_{r+3n},\nonumber\\
&&\sum^n_{j=0}p_1^jp_2^{n-j}(p_1^2+2p_2)^j(p_1^2+p_2)^{n-j}{n\choose j}\tilde F^{(2)}_{r+j}=\tilde F^{(2)}_{r+4n}.
\en
\end{corollary}

With an application of Corollary \ref{cor:5.2}, one may transfer the nonlinear expression (\ref{eq:0.3}) and its consequent identities shown in corollaries \ref{cor:0.3} and \ref{cor:0.4} to any linear recurring sequence defined by (\ref{eq:5.1}). For instance, from Corollary \ref{cor:0.4}, we immediately have 

\begin{corollary}\label{cor:0.5}
Let us consider $A_{2}$, the set of all linear recurring sequences defined by the homogeneous linear recurrence relation (\ref{eq:5.1}) with coefficient set $E_{2}=\{ p, q\}$. 
Then, for any $\{ a_n\}\in A_2$, there hold

\bns
&&\sum^n_{j=0}p_1^jp_2^{n-j}{n\choose j}(c a_{r+j-1}+da_{r+j-2})=ca_{r+2n-1}+d a_{r+2n-2},\nonumber\\
&&\sum^n_{j=0}(p_1^2+p_2)^j(p_1p_2)^{n-j}{n\choose j}(c a_{r+j-1}+da_{r+j-2})=ca_{r+3n-1}+d a_{r+3n-2},\nonumber\\
&&\sum^n_{j=0}p_1^jp_2^{n-j}(p_1^2+2p_2)^j(p_1^2+p_2)^{n-j}{n\choose j}\nonumber\\
&&\qquad  \times (c a_{r+j-1}+da_{r+j-2})=ca_{r+4n-1}+d a_{r+4n-2},
\ens
for $n\geq 0$, where $c$ and $d$ are given by 

\[
c= \frac{a_{1}-a_{0}p_1}{p_1(\tilde F^{(2)}_{1}-a_{0}a_{1}p_1-\tilde F^{(2)}_{0}p_2)},\quad d= -\frac{a_{1}p_2}{p_{1}(\tilde F^{(2)}_{1}-a_{0}a_{1}p_1-\tilde F^{(2)}_{0}p_2)},
\]
provided that $p_1\not= 0$, and $\tilde F^{(2)}_{1}-a_{0}a_{1}p_1-\tilde F^{(2)}_{0}p_2\not= 0$.
\end{corollary}

The nonlinear expression (\ref{eq:0.3}) can be used to obtain a congruence relations involving products of the IRSs as modules. 

\begin{corollary}\label{cor:0.6}
For $r\in{\bZ}$, $m\geq 1$, and $n\geq 0$, there holds a congruence relation of the form

\be\label{eq:0.6}
\tilde F^{(2)}_{mn+r}\equiv (p_2 \tilde F^{(2)}_{m-1})^n \tilde F^{(2)}_r+(\tilde F^{(2)}_m)^n \tilde F^{(2)}_{n+r}\,\,(mod\,\, \tilde F^{(2)}_{m-1} \tilde F^{(2)}_m).
\ee
In particular, for $r=0$ and $gcd\,(\tilde F^{(2)}_m,\tilde F^{(2)}_n)=1$, 

\be\label{eq:0.7}
\tilde F^{(2)}_{mn}\equiv 0\,\, (mod\,\, \tilde F^{(2)}_m \tilde F^{(2)}_n).
\ee
In general, if $\tilde F^{(2)}_{m_1}$, $\tilde F^{(2)}_{m_2},\ldots$, $\tilde F^{(2)}_{m_s}$ be relatively prime to each other with each $m_k\geq 1$ ($k=1,2,\ldots, s$), then there holds 

\be\label{eq:0.8}
\tilde F^{(2)}_{m_1m_2\cdots m_s}\equiv 0\,\, (mod\,\, \tilde F^{(2)}_{m_1}\tilde F^{(2)}_{m_2}\cdots \tilde F^{(2)}_{m_s}).
\ee
\end{corollary}

\begin{proof}
(\ref{eq:0.6}) comes from (\ref{eq:0.3}) straightforward. By setting $r=0$, we have 

\[
\tilde F^{(2)}_{mn}\equiv (\tilde F^{(2)}_m)^n \tilde F^{(2)}_{n}\,\, (mod\,\, \tilde F^{(2)}_{m-1} \tilde F^{(2)}_m)\equiv 0 \,\, (mod\,\, \tilde F^{(2)}_m).
\]
Similarly, 

\[
\tilde F^{(2)}_{mn}\equiv 0\,\, (mod\,\, \tilde F^{(2)}_n).
\]
Thus, if $gcd (\tilde F^{(2)}_m,\tilde F^{(2)}_n)=1$, i.e., $\tilde F^{(2)}_m$ and $\tilde F^{(2)}_n$ are relatively prime, then we obtain (\ref{eq:0.7}), which implies (\ref{eq:0.8}).
\end{proof}
\eop

\noindent{\bf Example 4.1} For $E_2=\{1,1\}, \{1,2\}, and \{2,1\}$, formula (\ref{eq:0.3}) in Theorem \ref{thm:0.2} leads the following three non-linear identities for Fibonacci, Pell, and Jacobsthal number sequences, respectively:

\bns
&&F_{mn+r}=\sum^n_{j=0} {n\choose j} F^j_mF^{n-j}_{m-1}F_{r+j},\\
&&P_{mn+r}=\sum^n_{j=0} {n\choose j} P^j_mP^{n-j}_{m-1}P_{r+j},\\
&&J_{mn+r}=\sum^n_{j=0} {n\choose j} J^j_m(2J_{m-1})^{n-j}J_{r+j},
\ens
where the first one is given in the main theorem of \cite{Hsu}. 

\medbreak
\noindent{\bf Example 4.2} As what we have presented, one may extend Fibonacci, Pell, and Jacobsthal numbers to negative indices as $\{F_n\}_{n\in {\bZ}}=\{\ldots, 2,-1,1,0,1,1,2,3,5,\ldots\}$, $\{P_n\}_{n\in {\bZ}}=\{ \ldots, 5,-2,1,0,1,2,5,12,29,\ldots\}$, and $\{J_n\}_{n\in {\bZ}}=\{ \ldots, 3/8, -1/4. 1/2, 0,1,$ $1,3,5,11,\ldots\}$ by using the corresponding linear recurrence relation with respect to $E_2=\{1,1,\}, \{ 2,1\}$, and $\{1,2\}$, respectively. Thus, from the first formula of (\ref{eq:0.4}) in Corollary \ref{cor:0.3}, there hold 

\bns
&&\sum^n_{j=0} {n\choose j} F^j_mF^{n-j}_{m-1}F_{j-mn-1}=1,\\
&&\sum^n_{j=0} {n\choose j} P^j_mP^{n-j}_{m-1}P_{j-mn-1}=1,\\
&&\sum^n_{j=0}2^{n-j+1} {n\choose j} J^j_mJ^{n-j}_{m-1}J_{j-mn-1}=1.
\ens
The identities generated by using the other formulas in (\ref{eq:0.4}) and the formulas in Corollaries \ref{cor:0.4}-\ref{cor:0.6} can be written similarly, which are omitted here. 

\section{Applications of IRS in Stirling numbers, Wythoff array, and the Boustrophedon transform} 
Let $A_{r}$ be the set of all linear recurring sequences defined by the homogeneous linear recurrence relation (\ref{eq:1}) with coefficient set $E_{r}=\{ p_{1}, p_{2}, \ldots, p_r\}$, and let $\tilde F^{( r)}_{n}$ be the IRS of $A_{r}$. Theorems \ref{thm:1.4} gives expressions of $\{a_{n}\}\in A_{r}$ and the IRS of $A_{r}$ in terms of each other. We will give a different expression of $\{a_{n}\}\in A_{r}$ in terms of $\tilde F^{( r)}_{n}$ using the generating functions of $\{a_{n}\}$ and $\{ \tilde F^{( r)}_{n}\}$ shown in Proposition \ref{pro:3.1}. More precisely, let $\{ a_{n}\}\in A_{r}$. Then, Proposition \ref{pro:3.1} shows that its generating function $P(t)$ can be written as (\ref{eq:3.1}). In particular, the generating function for the IRS with respect to $\{ p_j\}$ is presented as (\ref{eq:3.2}). Comparing the coefficients of $P(t)$ and $\tilde P(t)$, one may have 

\begin{proposition}\label{pro:3.0}
Let $\{ a_{n}\}$ be an element of $A_{r}$, and let $\{ \tilde F^{( r)}_{n}\}$ be the IRS of $A_{r}$. Then 

\bn
&&\tilde F^{( r)}_{n}=[t^{n-r+1}]\frac{1}{1-\sum^{r}_{j=1}p_{j}t^{j}}\label{eq:3.3}\\
&&a_{n}=a_{0}\tilde F^{( r)}_{n+r-1}+\sum^{r-1}_{k=1}(a_{k}-\sum^{k}_{j=1}p_{j}a_{k-j})\tilde F^{( r)}_{n+r-k+1}.\label{eq:3.3.2}
\en
\end{proposition}

\begin{proof}
(\ref{eq:3.3}) comes from 

\[
\tilde F^{( r)}_{n}=[t^{n}]\tilde P(t)=[t^{n-r+1}]\frac{1}{1-\sum^{r}_{j=1}p_{j}t^{j}}.
\]
Hence,  

\bns
&&a_{n}=[t^{n}]P(t)=a_{0}[t^{n}]\frac{1}{1-\sum^{r}_{j=1}p_{j}t^{j}}+[t^{n}]\frac{\sum^{r-1}_{k=1}\left( a_{k}-\sum^{k}_{j=1}p_{j}a_{k-j}\right)t^{k}}{1-\sum^{r}_{j=1}p_{j}t^{j}}\\
&=&a_{0}\tilde F^{( r)}_{n+r-1}+\sum^{r-1}_{k=1}\left( a_{k}-\sum^{k}_{j=1}p_{j}a_{k-j}\right)[t^{n-k}] \frac{1}{1-\sum^{r}_{j=1}p_{j}t^{j}},
\ens
which implies (\ref{eq:3.3}).
\end{proof}
\eop

\begin{proposition}\label{pro:3.2}
If 

\be\label{eq:3.-1}
g(t)=\frac{ t^{r-1}}{1-\sum^{r}_{j=1}p_{j}t^{j}}
\ee
is the generating function of a sequence $\{ b_n\}$, then $\{ b_n\}$ satisfies 

\[
b_n=\sum^r_{j=1}p_j b_{n-j}, \quad n\geq r,
\]
$b_0=\cdots =b_{r-2}=0$ and $b_{r-1}=1$.
\end{proposition}

\begin{proof}
It is clearly that

\bns
&&b_j=[t^j] g(t)=[t^{j-r+1}]\frac{ 1}{1-\sum^{r}_{j=1}p_{j}t^{j}}\\
&=&[t^{j-r+1}]\sum_{k\geq 0} \left( \sum^{r}_{j=1}p_{j}t^{j}\right)^k=\left\{ \begin{array}{ll} 0 & for\,\, 0\leq j\leq r-2,\\
1 &if\,\, j=r-1.\end{array}\right.
\ens
For $n\geq r$, 

\bns
&&\sum^r_{j=1} p_j b_{n-j}=\sum^r_{j=1} p_j [t^{n-j}]g(t)\\
&=&[t^n]\sum^r_{j=1}t^{r-1} \frac{ p_jt^{j}}{1-\sum^{r}_{i=1}p_{i}t^{i}}\\
&=&[t^n] t^{r-1} \left( \frac{ 1}{1-\sum^{r}_{i=1}p_{i}t^{i}}-1\right)\\
&=&[t^n]g(t)-[t^n]t^{r-1}=b_n,
\ens
which completes the proof.
\end{proof}
\eop

The Stirling numbers of the second kind $S(n,k)$ count the number of ways to partition a set of $n$ labelled objects into $k$ unlabeled subsets. We now prove sequence $\{b_n =S(n+1,k)\}_{n\geq 0}$ is the IRS of $A_k$ for any arbitrarily $k\in{\bN}_0$ by using the structure of the triangle of the Stirling numbers of the second kind:

\begin{table} \centering
$\begin{array}{c|rrrrrrr}
  n\backslash k & 0 & 1 & 2 & 3 & 4 & 5 & 6 \\ \hline
  0 & 1  \\
  1 &  0& 1  \\
  2 & 0& 1 & 1  \\
  3 & 0& 1 & 3 & 1  \\
  4 &  0 & 1 & 7 & 6& 1  \\
  5 & 0 & 1& 15 & 25 & 10 & 1  \\
  6 &  0 & 1 &  31 & 90 & 65 & 15 & 1 \\
\end{array}$
\caption{Table $1$. Triangle of the Stirling numbers of the second kind}
\end{table}

\begin{theorem}\label{thm:3.3}
Let $S(n,k)$ be the Stirling numbers of the second kind. Then, for any fixed integer $k>0$, $S(n+1,k)$ is the IRS of $A_k$, i.e., 

\be\label{eq:3.4}
S(n+1,k)=\tilde F^{(k)}_n,
\ee
where $A_k$ is the set of linear recurring sequences $\{ a_n\}$ generated by 

\be\label{eq:3.5}
a_n=\sum^k_{j=1}p_j a_{n-j},
\ee
which has characteristic polynomial 

\be\label{eq:3.6}
P(t)=t^k-\sum^k_{j=1}p_j t^{k-j}=\Pi^k_{j=1} (1-j t).
\ee
Thus, the generating function of $\{ S(n+1,k)\}$ is 

\be\label{eq:3.7}
\tilde P(t)=\frac{t^k}{\Pi^k_{j=1} (1-jt)}.
\ee
\end{theorem}

\begin{proof}
$\{S(n,k)\}_{0\leq k\leq n}$ obeys the recurrence relation 

\be\label{eq:3.8}
S(n+1,k)=k S(n,k)+S(n,k-1), \quad 0\leq k\leq n, 
\ee
which can be explained as follows from the definition of $S(n,k)$ shown above: A partition of the $n+1$  objects into $k$ nonempty subsets either contains the subset $\{ n+1\}$ or it does not. The number of ways that $\{ n+1\}$ is one of the subsets is given by $S(n,k-1)$, since we must partition the remaining  $n$ objects into the available $k-1$ subsets. The number of ways that $\{ n+1\}$ is not one of the subsets  is given by $k S(n,k)$ because we partition all elements other than $n+1$ into $k$ subsets with $k$  choices for inserting the element $n+1$. Summing these two values yields (\ref{eq:3.8}). Write (\ref{eq:3.8}) as 

\[
S(n,k-1)=S(n+1,k)-kS(n,k), \quad 0< k\leq n.
\]
Denote by $g_{k-1}(t)$ and $g_k(t)$ the generating functions of $k-1$st and $k$th columns of the Stirling triangle shown in Table $2$. Then the above equation means 

\[
[t^n] g_{k-1}(t)=S(n-1,k)=S(n+1,k)-kS(n,k)=[t^{n+1}]g_k(t)-k [t^n]g_k(t)=[t^{n+1}](1-kt)g_k(t)
\]
for $0< k\leq n$. Therefore, we have 

\[
g_k(t)=\frac{t}{1-kt}g_{k-1}(t).
\]
Since $g_1(t)=t/(1-t)$, we obtain the generating function of $\{ S(n,k)\}$ for $k>0$ 

\be\label{eq:3.9}
g_k(t)=\frac{t^{k}}{\Pi^k_{j=1} (1-jt)}.
\ee
Denote $b_n=S(n+1,k)$ and the generating function of $\{ b_n\}$ by $B(t)$. Then 

\[
b_n=[t^n] B(t)=S(n+1,k)=[t^{n+1}] g_k(t), \quad 0<k\leq n,
\]
which implies the generating function of $\{ b_n\}$ is 

\be\label{eq:3.10}
B(t)=\frac{1}{t} g_k(t)=\frac{t^{k-1}}{\Pi^k_{j=1} (1-jt)}.
\ee
Using Proposition \ref{pro:3.2}, one immediately know that $\{ S(n+1,k)\}$ is the IRS of $A_k$ that has the generating function (\ref{eq:3.7}), where $A_k$ is the set of linear recurring sequences $\{ a_n\}$ generated by recurrence relation (\ref{eq:3.5}) with the characteristic polynomial (\ref{eq:3.6}).
\end{proof}
\eop

\noindent{\bf Example 5.1} Since the second column of the Stirling triangle, $\{ S(n+1, 2)\}$, has the generating function 

\[
\tilde P_2(t)=\frac{t}{(1-t)(1-2t)}=\frac{1}{1-3t+2t^2},
\]
there holds recurrence relation 

\[
S(n+1,2)=3S(n,2)-2S(n-1,2)
\]
with the initial $S(1,2)=0$ and $S(2,2)=1$. Thus we have $S(3,2)=3$, $S(4,2)=7$, $S(5,2)=15$, $S(6,2)=31$, etc.

Similarly, $\{ S(n+1,3)\}$ is a linear recurring sequence generated by 

\[
S(n+1,3)=6S(n,3)-11S(n,2)+6S(n,1)
\]
with initials $S(1,3)=S(2,3)=0$ and $S(3,3)=1$. From the recurrence relation one may find $S(4,3)=6$, $S(5,3)=25$, $S(6,3)=90$, etc.

We now consider the Wythoff array shown in Table $2$, in which the two columns to the left of the vertical line consist respectively of the nonnegative integers n, and the lower Wythoff sequence (A201), whose $n$th term is $[(n + 1)\alpha]$, where $\alpha=(1+\sqrt{5})/2$. The rows are linear recurring sequences generated by the Fibonacci rule that each term is the sum of the two previous terms.  

\begin{center}
\begin{tabular}{cccccccccc}
  $0$ & $1$ & $\vline$ & $1$& $2$ & $3$ & $5$& $8$ & $13$&$\cdots$\\
 \rule[-3mm]{0mm}{8mm}
  $1$ & $3$& $\vline $ & $4$& $7$&$11$& $18$ & $29$ &$47$&$\cdots$\\ 
 \rule[-3mm]{0mm}{8mm}
 $2$ & $4$ & $\vline$ & $6$& $10$ & $16$ & $26$& $42$ & $68$&$\cdots$\\
 \rule[-3mm]{0mm}{8mm}
 $3$ & $6$ & $\vline$ & $9$& $15$ & $24$ & $39$& $63$ &$102$ & $\cdots$\\
 \rule[-3mm]{0mm}{8mm}
 $4$ & $8$ & $\vline$ & $12$& $20$ & $32$ & $52$& $84$& $136$ &$\cdots$  \\
 \rule[-3mm]{0mm}{8mm}
 $5$ & $9$ & $\vline$ & $14$& $23$ & $37$ & $60$& $97$ &$157$ &$\cdots$  \\
 \rule[-3mm]{0mm}{8mm}
$6$ & $11$ & $\vline$ & $17$& $28$ & $45$ & $73$& $118$ & $191$ &$\cdots$  \\
 \rule[-3mm]{0mm}{8mm}
$7$ & $12$ & $\vline$ & $19$& $31$ & $50$ & $81$ &$131$ & $212$ & $\cdots$  \\
 \rule[-3mm]{0mm}{8mm}
$\vdots$ &$\vdots$  &$\vline$ & $\vdots$& $\vdots$&$\vdots$ & $\vdots$&$\vdots$&$\vdots$
 \end{tabular}
 \centerline{Table $2$. the Wythoff array}
 \end{center}

The first row sequence in Table $2$ is Fibonacci sequence, i.e., the IRS in $A_2$  with respect to $E_2=\{1,1\}$. The $j$th row sequence $\{ a^{(j)}_n\}_{n\geq 0}$ are linear recurring sequences in $A_2$ with respect to $E_2=\{1,1\}$ with initials $(j, [(j+1)\alpha])$, $j=1,2,\ldots$, where $\alpha=(1+\sqrt{5})/2$. From (\ref{eq:5.3-0}) or (\ref{eq:5.1-2}), the expression of $a_n^{(j)}$ can be written as 

\bn\label{eq:3.11}
a_n^{(j)}&=&\left( \frac{\left[(j+1)\frac{1+\sqrt{5}}{2}\right]-j \frac{1-\sqrt{5}}{2}}{\sqrt{5}}\right) \left( \frac{1+\sqrt{5}}{2}\right)^n\nonumber \\
&&- \left(\frac{\left[(j+1)\frac{1+\sqrt{5}}{2}\right]-j \frac{1+\sqrt{5}}{2}}{\sqrt{5}}\right) \left( \frac{1-\sqrt{5}}{2}\right)^n.
\en
Since every integer in ${\bN}_0$ appears exactly once (see \cite{Slo}), the collection $\{ a^{(j)}_n\}_{n\geq 0}$ ($j=0,1,\ldots$) gives a partition of ${\bN}_0$. $[(j+1)\alpha]$, $\alpha=(1+\sqrt{5})/2$, can be considered as the representative of the equivalence class $\{ a^{(j)}_n: n=0,1,\ldots\}$.  

The Wythoff array can be extended to similar arrays associated with different linear recurrence relations. For instance, one may define the Pell-Wythoff array shown in Table $3$.

\begin{center}
\begin{tabular}{cccccccccc}
  $0$ & $1$ & $\vline$ & $2$& $5$ & $12$ & $29$& $70$ & $169$&$\cdots$\\
 \rule[-3mm]{0mm}{8mm}
  $1$ & $3$& $\vline $ & $7$& $17$&$41$& $99$ & $239$ &$577$&$\cdots$\\ 
 \rule[-3mm]{0mm}{8mm}
 $2$ & $6$ & $\vline$ & $14$& $34$ & $82$ & $198$& $470$ & $1154$&$\cdots$\\
 \rule[-3mm]{0mm}{8mm}
 $3$ & $8$ & $\vline$ & $19$& $46$ & $111$ & $268$& $647$ &$1562$ & $\cdots$\\
 \rule[-3mm]{0mm}{8mm}
 $4$ & $11$ & $\vline$ & $26$& $63$ & $152$ & $367$& $886$& $2139$ &$\cdots$  \\
 \rule[-3mm]{0mm}{8mm}
 $\vdots$ &$\vdots$  &$\vline$ & $\vdots$& $\vdots$&$\vdots$ & $\vdots$&$\vdots$&$\vdots$
 \end{tabular}
 \centerline{Table $3$. the Pell-Wythoff array}
 \end{center}
 The first row sequence in Table $2$ is the Pell sequence, which is also the IRS in $A_2$  with respect to $E_2=\{2,1\}$. The $j$th row sequence $\{ b^{(j)}_n\}_{n\geq 0}$ are linear recurring sequences in $A_2$ with respect to $E_2=\{2,1\}$ with initials $(j, [(j+1)r]-1)$, $j=1,2,\ldots$, where $r=1+\sqrt{2}$. 

From (\ref{eq:5.3-0}) or (\ref{eq:5.1-2}), we obtain the expression of $b_n^{(j)}$ as 

\bn\label{eq:3.12}
b_n^{(j)}&=&\left( \frac{\left[(j+1)(1+\sqrt{2})\right]-j (1-\sqrt{2})-1}{2\sqrt{2}}\right) \left( 1+\sqrt{2}\right)^n\nonumber \\
&&- \left(\frac{\left[(j+1)(1+\sqrt{2})\right]-j (1+\sqrt{2})-1}{2\sqrt{2}}\right) \left( 1-\sqrt{2}\right)^n.
\en

From \cite{MSY}, the boustrophedon transform of a given sequence $\{ a_n\}_{n\geq 0}$ is the sequence produced by the triangle shown in Table $4$.

\begin{table} \centering
$\begin{array}{llllllll}
 a_0=b_0\\
  a_1\rightarrow  & b_1=a_0+a_1  \\
  b_2=a_1+a_2+b_1 &  \leftarrow a_2+b_1& \leftarrow a_2  \\
  a_3 \rightarrow & a_3+b_2 \rightarrow & a_2+a_3+b_1+b_2\rightarrow & b_3=2a_2+a_3+b_1+b_2  \\
 \end{array}$
\medbreak
\centerline{Table $4$. Triangle of the boustrophedon transform}
\end{table}
Formally, the entries $T_{n,k}$ ($n\geq k\geq 0$) of the above triangle are defined by 

\bn\label{eq:3.13}
&&T_{n,0}=a_n \quad n\geq 0,\nonumber \\
&&T_{n+1,k+1}=T_{n+1,k}+T_{n,n-k} \quad n\geq k\geq 0,
\en
and then $b_n=T_{n,n}$ for all $n\geq 0$. Denote the (exponential) generating functions of sequences $\{ a_n\}$ and $\{ b_n\}$ by $A(t)$ and $B(t)$. Then Theorem $1$ of \cite{MSY} gives 

\[
B(t)=(\sec t+\tan t) A(t),
\]
which implies a relationship of the generating function $\tilde P_r (t)$ of the IRS $\{ \tilde F^{(r)}_n\}$ of $A_r$ and the generating function $(BP_r)(t)$ of the boustrophedon transform sequence of $\{\tilde F^{(r)}_n\}$:

\be\label{eq:3.14}
(BP_r)(t)=(\sec t+\tan t) P_r(t)=(\sec t+\tan t)\frac{t^{r-1}}{1-\sum^r_{j=1}p_j t^j}.
\ee

\end{document}